\documentclass[11pt]{amsart}
\usepackage{amsmath, amsthm, amssymb}
\newcommand{\mb}{\mathbb}
\newcommand{\mc}{\mathcal}
\newcommand{\p}{\partial}

\usepackage{amsfonts}

\newtheorem{theorem}{Theorem}[section]
\newtheorem{definition}[theorem]{Definiton}
\newtheorem{lemma}[theorem]{Lemma}

\begin{document}

\title{Pinching on open manifolds}

\author{Manuel Streil}

\address{Institut f\"{u}r Mathematik\\Universit\"{a}t Augsburg\\86135
Augsburg\\Germany}

\email{manuel\_streil@web.de}

\begin{abstract}
We show that the $2$-jet bundle of local Riemannian metrics on an arbitrary differentiable 
manifold admits a section which pointwise fulfills the curvature relation $\sec(g)=a$ for any
$a\in\mb{R}.$ It follows by Gromov's h-principle for open, invariant differential 
relations that every noncompact differentiable manifold carries arbitrarily pinched 
(incomplete) Riemannian metrics.
\end{abstract}
\maketitle

\section{Introduction}
A basic question in Riemannian geometry is what effect the existence
of a Riemannian metric with particular curvature properties has on the
topology of the underlying manifold. One usually considers complete metrics
and obtains even on an elementary level rather strong restrictions.
For instance, complete manifolds of negative sectional curvature are always
aspherical and all nontrivial elements of the fundamental group have infinite
order. Furthermore a closed manifold does not support two metrics of 
different signed sectional curvatures. We also mention the result of Gromoll
and Meyer \cite{GM} that a complete open manifold of positive sectional curvature is
diffeomorphic to some $\mb{R}^m.$

The situation changes completely if we skip the completeness assumption, which
is automatically imposed on metrics on closed manifolds but is an extra 
condition on open manifolds. Gromov remarks in his thesis \cite{G1} the astonishing fact
that every open manifold carries metrics with strict positive and negative 
sectional curvature.

In this paper we extend this result and prove the existence of arbitrarily
pinched (incomplete) metrics on open manifolds.

\begin{theorem} \label{pinch}
Let $M$ be an open manifold. Given $\delta >0$ and $a\in\mb{R}$ there 
exists a Riemannian metric on $M$ such that all sectional curvatures 
are in $(a-\delta,a+\delta).$
\end{theorem}

In other words the geometric significance of the sign of the curvature  and even 
of curvature bounds depends on the completeness of the underlying metric. It turns out 
that the only question concerning curvature which involves geometry in the incomplete 
case is the existence of metrics of constant sectional curvature, since then topological
obstructions are known. For example such manifolds have trivial Pontryagin classes
(cf. Theorem 43 and Corollary 44 in \cite{Sp}).

Theorem \ref{pinch} is an application of a deep differential topological insight of M. Gromov
who greatly generalized in his thesis \cite{G1} Smale-Hirsch-Phillips immersion-submersion theory 
(see \cite{S}, \cite{H}, \cite{Ph}) by proving what is now called the h-principle for 
invariant, open differential relations on open manifolds. 

Roughly speaking, a partial differential relation $\mc{R}$ is any condition 
imposed on the partial derivatives of an unknown function. By substituting derivatives 
by new independent variables one gets an underlying algebraic relation. 
Obviously the existence of a formal solution, i. e. a solution of the corresponding 
algebraic relation, is a necessary condition for the solvability of $\mc{R}.$ It turns out
that under certain cirumstances any formal solution of $\mc{R}$ can be deformed into a
genuine one. 

In general one considers a smooth fibre bundle over the underlying manifold. 
A partial differential relation is a condition imposed on the $r-$jet bundle of 
local sections. Now we could try and construct a section of this $r-$jet bundle which
pointwise fulfills the relation, i.e. a formal solution, and deform it to get a real solution. 
In other words, we reduce the problem to algebraic-topological obstruction theory. 
We refer to \cite{EM} and \cite{G2} for expositions of this technique.

In our case we study the bundle of symmetric bilinear forms the positive definite sections of which
are Riemannian metrics on the underlying manifold. The observation that the sectional
curvature only depends on the $2-$jet of the metric allows us to translate the pinching
problem into a curvature relation, see Definition \ref{rel}. In Lemma \ref{formal} we will prove that 
on any manifold there are formal solutions with {\it constant} curvature. Then, as an application 
of Gromov's h-principle, Theorem \ref{h} implies the existence of arbitrarily pinched metrics.

\section{Some basic facts concerning jets}

Given a $(C^{\infty}$-)smooth fibre bundle $q: V\rightarrow M$ over a smooth manifold $M$ of dimension $m$ we
identify two local smooth sections $\sigma_1$ and $\sigma_2$ defined in a neighbourhood of some point
$p\in M$ if in local local coordinates on $M$ and $V$ they have the same partial derivatives
up to order $r$ at $p.$ An equivalence class $[(\sigma,p)]=j^r_p \sigma$ under this relation 
is called $r$-jet of $\sigma$ at $p.$ We denote by $V^r$ the space of $r$-jets of local sections.
A partial differential relation of order $r$ is a subset $\mc{R}\subset V^r.$

In the following we consider the bundle $q:E\rightarrow M$ of the positive definite 2-forms on $M,$ which
is an open subbundle of the bundle $q: S^2 T^\ast M\rightarrow M$ of the symmetric bilinear forms
on $M.$ Having chosen a chart neighbourhood $U\subset M$ we have trivialisations 
$q^{-1}(U)\cong U\times S(m)$ and  $q^{-1}(U)\cong U\times P(m)$ where we denote by 
$S(m)\subset \mb{R}^{m\times m}$ the vectorspace ($\dim S(m)=:d$) of symmetric $m\times m$ 
matrices and by $P(m)\subset S(m)$ the open subset of positive definite $m\times m$ matrices.  

We consider the projection map
\begin{eqnarray*} 
q^2: E^2 & \rightarrow & M\\
j^2_p g & \mapsto & p.
\end{eqnarray*}
Let $p\in M$ and $j^2_p g\in E^2$ a 2-jet represented by a local Riemannian metric $g.$ 
Using a chart $(U,\Phi,x^1,...,x^m)$ around $p,$ we obtain a local description 
\[(g_{ij}^\Phi(\Phi(p)),\frac{\p g_{ij}^\Phi }{\p x^k}(\Phi(p)),\frac{\p^2 g_{ij}^\Phi}{\p x^l\p x^k} (\Phi(p)))
\in P(m)^2:=P(m)\times\mb{R}^{dm}\times\mb{R}^{\frac{dm(m+1)}{2}}\]
of $j^2_p g$ taking symmetries of the partial derivatives into account. Vice versa, given
$(g_{ij}^\Phi,g_{ijk}^\Phi,g_{ijkl}^\Phi)\in P(m)^2$ we choose the associated taylor polynom
$h_{ij}$ near $\Phi(p)$ such that 
\[(h_{ij}(\Phi(p)),\frac{\p h_{ij} }{\p x^k}(\Phi(p)),\frac{\p^2 h_{ij}}{\p x^l\p x^k} (\Phi(p)))=
(g_{ij}^\Phi,g_{ijk}^\Phi,g_{ijkl}^\Phi).\]
As a consequence there is a 1-1 correspondence 
\[(q^2)^{-1}(U)\stackrel{1-1}{\longleftrightarrow}U\times P(m)^2.\]
Now let $(W,\Psi,y^1,...,y^m)$ be another chart defined near $p.$ Keeping in mind that
\[g_{ij}^\Psi=\sum_{k,l}\frac{\p x^k}{\p y^i}\frac{\p x^l}{\p y^j}g_{kl}^\Phi\]
a change of coordinates on $M$ obviously induces a linear transformation of $P(m)^2.$ This allows us to give
$E^2$ a canonical structure of a smooth manifold and we have proven

\begin{lemma}
The space $E^2$ of 2-jets of local Riemannian metrics on $M$ defines a smooth fibre bundle
\begin{eqnarray*} 
q^2: E^2 & \rightarrow & M\\
j^2_p g & \mapsto & p
\end{eqnarray*}
with linear transformations of the fibre 
$P(m)^2=P(m)\times\mb{R}^{dm}\times\mb{R}^{\frac{dm(m+1)}{2}}.$
\end{lemma}

\section{The curvature relation}
Let $\tau=j^2_p g$ be a 2-jet of a local Riemannian metric $g$ at $p\in M$ and $(U,\Phi,x^1,...,x^m)$ a chart around $p.$
We may assume that $g$ is defined on $U.$ In local coordinates $\tau$ is associated to
\[(\Phi(p),\tau_{ij},\tau_{ijk},\tau_{ijkl})=
(\Phi(p),g^\Phi_{ij}(\Phi(p)),\p_k g^\Phi_{ij}(\Phi(p)),\p_l\p_k g^\Phi_{ij}(\Phi(p))).\]
On $U$ the metric $g$ induces the Levi-Civita connection $\nabla$ given locally by means of the 
Christoffelsymbols defined by
\[\nabla_{\frac{\p}{\p x^i}}\frac{\p}{\p x^j} = \sum_{k=1}^m \Gamma_{ij}^k\cdot\frac{\p}{\p x^k}\]
or more explicitly
\[\Gamma_{ij}^k=\frac{1}{2}\sum_{l=1}^m (g^\Phi)^{lk}(\p_i g^\Phi_{jl}+\p_j g^\Phi_{li}-\p_l g^\Phi_{ij})\]
where $((g^\Phi)^{ij})$ is the inverse of the matrix $(g^\Phi_{ij}).$ We note that the entries $(g^\Phi)^{ij}$ are 
rational functions in terms of the $g^\Phi_{ij}.$ It follows that the components $R_{ijks}$ of the curvature tensor 
given by
\[g\left(R\left(\frac{\p}{\p x^i},\frac{\p}{\p x^j}\right)\frac{\p}{\p x^k},\frac{\p}{\p x^s}\right)=
g\left(\nabla_\frac{\p}{\p x^j}\nabla_\frac{\p}{\p x^i}\frac{\p}{\p x^k}
-\nabla_\frac{\p}{\p x^i}\nabla_\frac{\p}{\p x^j}\frac{\p}{\p x^k},\frac{\p}{\p x^s}\right)\]
at $p$ are completely determined by $\tau$ and independent of the choice of the representative $g.$
This implies that the sectional curvature 
\[\sec(V)=\frac{g(R(v,w)v,w)}{g(v,v)\cdot g(w,w)-g(v,w)^2}\]
of a plane $V\subset T_p M$ spanned by two linearly independent vectors $v,w\in T_p M$ only depends
on $\tau.$ In other words, we have a well-defined notion of sectional curvature of 2-jets of local
Riemannian metrics. 

\begin{definition}
Let $\tau\in E^2$ be a 2-jet at $p\in M$ of a Riemannian metric $g$ on a neighbourhood $U$ of $p$ and 
$V$ a plane in $T_p M.$ We define the sectional curvature $\sec_\tau V$ as the sectional curvature 
$\sec_{(U,g)}V$ with respect to the Levi-Civita connection induced by $g$ on $U.$ 
\end{definition}

\begin{definition} \label{rel}
Given $a\in\mb{R}$ and $\delta \geq 0$ the curvature relation $\mc{R}_{\delta,a}$ is defined as the subset
\begin{eqnarray*}
\mc{R}_{\delta,a}= \{\tau\in E^2: \sec_\tau V\in (a-\delta,a+\delta)\, \mbox{for all planes}\, V\subset T_{q^2(\tau)}M\}
\end{eqnarray*}
where we think of $(a,a)$ as $\{a\}$ if $\delta=0$ by abuse of notation.
\end{definition}
We call a continuous section $\alpha$ of $E^2$ satisfying $\alpha(M)\subset\mc{R}_{\delta,a}$ a {\bf formal} solution to
the curvature relation $\mc{R}_{\delta,a}.$ A formal solution $\alpha$ is {\bf holonomic} if $j^2 g=\alpha$ for some section $g$ of $E.$
In other words, $g$ is a Riemannian metric on $M$ such that all sectional curvatures lie in $(a-\delta,a+\delta).$
We refer to \cite{EM} and \cite{G2} for these notions in the broader context of partial differential relations.

Let $M_1$ and $M_2$ be two differentiable manifolds of the same dimension. As above we define bundles
$q_i: E_i\rightarrow M_i$ and $q^2_i:E^2_i\rightarrow M_i$ as well as curvature relations $\mc{R}^i_{\delta,a}$
with $i\in\{1,2\}.$ A local diffeomorphism $f:U_1\rightarrow U_2$ between open subsets $U_1\subset M_1$ and $U_2\subset M_2$
induces in a canonical way a map 
\[f_\ast:E^2_1|U_1\rightarrow E^2_2|U_2\] 
as follows: Let $\tau=j^2_p g\in E^2_1.$ We may assume that
its representative $g$ is defined on $U_1.$ The push-forward $f_\ast g$ yields a metric on $U_2$ via
\[f_\ast g(v,w)=g(f^\ast v,f^\ast w)\]
for all $v,w \in T_{\tilde{p}} M_2$ and arbitrary $\tilde{p} \in U_2.$ Now we define
\[f_\ast \tau :=j^2_{f(p)}f_\ast g.\]
One readily checks that $f_\ast\tau$ is well-defined. 
Suppose that $\tau=j^2_p g\in\mc{R}^1_{\delta,a}.$ Due to the fact that $f:(U_1,g)\rightarrow (U_2,f_\ast g)$ is an
isometry we have
\[\sec_{f_\ast\tau}E(v,w)=\sec_\tau E(f^\ast v,f^\ast w)\]
where $E(v,w)$ is the plane spanned by two linearly independent vectors $v,w\in T_{f(p)}U_2.$ 
In other words, $f_\ast\tau\in\mc{R}^2_{\delta,a}.$ In this sense curvature relations are invariant under local
diffeomorphisms.
\begin{lemma}
The restriction $q^2:\mc{R}_{\delta,a}\rightarrow M$ defines a subbundle of $q^2:E^2\rightarrow M.$ 
This bundle is open if $\delta >0.$
\end{lemma}
\begin{proof}
Let $p\in M$ and $(U,\Phi)$ a chart of $M$ near $p.$ We obtain a local trivialisation
$(q^2)^{-1}(U)\cong U\times P(m)^2$ of the bundle $q^2: E^2\rightarrow M.$ Any 2-jet $\tau\in\mc{R}_{\delta,a}$ has a local
representation
\[(p,\tau_{ij},\tau_{ijk},\tau_{ijkl})\in U\times P(m)^2.\]
We note that the sectional curvature depends only on $(\tau_{ij},\tau_{ijk},\tau_{ijkl}).$ 
Consequently, $\tilde{\tau}=(\tilde{p},\tau_{ij},\tau_{ijk},\tau_{ijkl})\in\mc{R}_{\delta,a}$ for any $\tilde{p}\in U.$
We define $F_{\delta,a}$ as a subset of $P(m)^2$ in the following way: $(\tau_{ij},\tau_{ijk},\tau_{ijkl})\in F_{\delta,a}$ if and
only if for some (and hence any) $\tilde{p}\in U$ we have
\[\tilde{\tau}=(\tilde{p},\tau_{ij},\tau_{ijk},\tau_{ijkl})\in\mc{R}_{\delta,a}.\]
In other words, 
\[(q^2)^{-1}(U)\cap \mc{R}_{\delta,a}\cong U\times F_{\delta,a}.\]
As a result, $q^2:\mc{R}_{\delta,a}\rightarrow M$ is trivial over charts of $M.$ A change of coodinates on $M$ induces
a linear transformation of $F_{\delta,a}\subset P(m)^2,$ because the sectional curvature is independent of the choice of
local coordinates.\\
Now assume $\delta>0.$ We identify $T_p M$ with $\mb{R}^m$ and parametrize the set of all planes in $T_p M$ by the
Stiefel manifold $V_{m,2}$ of the orthonormal 2-frames $(v,w)\in\mb{R}^m.$ Furthermore, let $E(v,w)$ be the plane in $\mb{R}^m$
spanned by two linearly independent vectors $v$ and $w.$
The function 
\begin{eqnarray*}
\eta: P(m)^2\times V_{m,2} & \rightarrow & \mb{R}\\
(\sigma',(v,w)) & \mapsto & \sec_{(p,\sigma)}(E(v,w))
\end{eqnarray*}
is continuous and $V_{m,2}$ is compact. It follows that $F_{\delta,a}$ is an open subset of $P(m)^2.$
\end{proof}
Let $\delta >0.$ We denote by $\Gamma\mc{R}_{\delta,a}$ the space of formal solutions equipped with the
compact-open topology and write 
\[\Gamma_{\mc{R}_{\delta,a}}E=\{g\in\Gamma^\infty(E):j^2 g\in\mc{R}_{\delta,a}\}\]
for the space of smooth sections of E, i.e. Riemannian metrics on M, such that all sectional curvatures lie
in $(a-\delta,a+\delta).$ 
The map $j^2:\Gamma_{\mc{R}_{\delta,a}}E\rightarrow\Gamma\mc{R}_{\delta,a}$ induces the weak $C^2-$topology on
$\Gamma_{\mc{R}_{\delta,a}}E.$

So far we have discussed all technical notions we need to apply Gromov's general h-principle for open,
invariant relations on open manifolds to our special case of curvature relations (cf. Theorem 3.12 in \cite{A}):
\begin{theorem} \label{h}
Let $M$ be an open manifold and $\mc{R}_{\delta,a}$ a curvature relation with $\delta >0$. Then
\[j^2:\Gamma_{\mc{R}_{\delta,a}}E\rightarrow\Gamma\mc{R}_{\delta,a}\]
is a weak homotopy equivalence.
\end{theorem}

In other words, if we can prove that there exist formal solutions we are done. Surjectivity on $\pi_0$
then yields arbitrarily pinched Riemannian metrics.

\begin{lemma} \label{formal}
The space $\Gamma\mc{R}_{\delta,a}$ of formal solutions is nonempty for any $a\in\mb{R}$ and $\delta\geq 0.$
\end{lemma}
\begin{proof}
It suffices to show that $\mc{R}_{0,a}\subset\mc{R}_{\delta,a}$ admits formal solutions. We will prove
that the fibre $F_{0,a}=:F_a$ smoothly deformation retracts to an arbitrary element of $F_a.$

Note that $F_a$ is nonempty. In a neighbourhood of $0\in\mb{R}^m$ we choose a Riemannian metric $g$ with
constant sectional curvature $a.$ Then
\[\tau'=(\tau_{ij},\tau_{ijk},\tau_{ijkl})=(g_{ij}(0),\p_k g_{ij}(0),\p_l\p_k g_{ij}(0))\in F_a.\]
Let $\tau$ be the Taylor polynom which represents the 2-jet $(0,\tau')\in\mb{R}^m\times F_a.$ We
regard $\tau$ as a Riemannian metric defined in a neighbourhood of $0\in\mb{R}^m.$

Now we choose a reference basis $B=\{b_1,...,b_m\}$ of $\mb{R}^m.$ Applying the Gram-Schmidt procedure yields
an orthonormal basis $B_\tau$ w.r.t. $\tau(0).$ We change to normal coordinates centered at $0$ w.r.t. $\tau$
and $B_\tau,$ i.e. we identify $T_0\mb{R}^m$ with $\mb{R}^m$ via $B_\tau,$ and the exponential map induces a 
linear and invertible transformation
\begin{eqnarray*}
\tilde{L}: F_a & \rightarrow & F_a \\
           \tau' & \mapsto & \tilde{\tau}'=(\delta_{ij},0,\tilde{\tau}_{ijkl}),
\end{eqnarray*}
where $\delta_{ij}=1$ if $i=j$ and $\delta_{ij}=0$ if $i\neq j.$

Let $\psi'\in P(m)^2$ and $\psi$ the Taylor polynom associated to $(0,\psi')\in\mb{R}^m\times P(m)^2.$   
Like $\tau$ we think of $\psi$ as a Riemannian metric defined near $0\in\mb{R}^m.$ Then
\[\psi_t=t\cdot\tau+(1-t)\cdot\psi,\,t\in[0,1]\]
is a Riemannian metric in a neighbourhood of $0\in\mb{R}^m$ with $\psi_0=\psi$ und $\psi_1=\tau.$
We write     
\[\psi_t'=((\psi_t)_{ij}(0),\p_k (\psi_t)_{ij}(0),\p_l\p_k (\psi_t)_{ij}(0))\in  P(m)^2.\]
The Gram-Schmidt procedure w.r.t. $\psi_t(0)$ transforms $B$ into an orthonormal basis
\[B(t,\psi')=\{b_1(t,\psi'),...,b_m(t,\psi')\}\] with respect to $\psi_t(0)$ such that the maps
\begin{eqnarray*}
[0,1]\times P(m)^2 & \rightarrow & \mb{R}^m \\
(t,\psi') & \mapsto & b_i(t,\psi') \quad i=1,...,m
\end{eqnarray*}
are smooth and $B(1,\psi')=B_\tau.$ Changing to normal coordinates at $0$ w.r.t. $\psi_t$ and 
$B(t,\psi'),$ we obtain a family
\[L(t,\psi'):P(m)^2\rightarrow P(m)^2\]
of linear and invertible transformations with $L(1,\psi')=\tilde{L}.$ 

We claim that $L(t,\psi')$ depends smoothly on $(t,\psi')$
(and so does $L(t,\psi')^{-1}$ by Kramer's rule).

We observe that $(t,\psi',x)\mapsto\psi_t(x)$ is smooth. Let $(t_0,\phi')\in[0,1]\times P(m)^2,$
then $\phi_{t_0}(0)\in P(m).$ We find a neighbourhood $V$ of $(t_0,\phi')$ and a neighbourhood
$U$ of $0\in\mb{R}^m$ such that $\psi_t(x)\in P(m)$ if $(t,\psi',x)\in V\times U.$ 

The Christoffel symbols associated to $\psi_t$ depend smoothly on $(t,\psi')\in V,$ i.e.
\begin{eqnarray*}
V\times U & \rightarrow & \mb{R}\\
(t,\psi',x) & \mapsto & \Gamma_{ij}^k(t,\psi',x)
\end{eqnarray*}
is smooth. Thus we obtain a system
\[\ddot{x}^k+\sum_{i,j=1}^m \Gamma_{ij}^k(t,\psi',x)\dot{x}^i\dot{x}^j=0,\quad k=1,...,m\]
of geodesic equations. 
Hence there exist neighbourhoods $V'$ of $(t_0,\phi')$ and $W$ of $0\in T_0\mb{R}^m$
such that 
\begin{eqnarray*}
V'\times W & \rightarrow & \mb{R}^m\\
(t,\psi',v) & \mapsto & \exp_0(t,\psi',v)
\end{eqnarray*}
is smooth and our claim follows.

We write $\tilde{\psi}'=L(0,\psi')(\psi')$ and define a smooth map
\begin{eqnarray*}
h: [0,1]\times P(m)^2 & \rightarrow & P(m)^2\\
        (s,\psi') & \mapsto & s\cdot\tilde{\tau}'+(1-s)\cdot\tilde{\psi}'.
\end{eqnarray*}
We set
\begin{eqnarray*}
G: [0,1]\times P(m)^2 & \rightarrow & P(m)^2\\
    (t,\psi') & \mapsto & L(t,\psi')^{-1}(h(t,\psi'))
\end{eqnarray*}
and obtain a smooth map which satisfies
\[G(0,\psi')=L(0,\psi')^{-1}(h(0,\psi'))=L(0,\psi')^{-1}(\tilde{\psi}')=\psi'\]
and
\[G(1,\psi')=L(1,\psi')^{-1}(h(1,\psi'))=\tilde{L}^{-1}(\tilde{\tau}')=\tau'.\]
In case $\psi'=\tau'$ it follows that $\psi_t=\tau$ and $L(t,\tau')=\tilde{L}$ 
independent of $t\in[0,1].$ Thus, $h(s,\tau')=\tilde{\tau}'$ for all $s\in[0,1]$ and
$G(t,\tau')=\tilde{L}^{-1}(\tilde{\tau}')=\tau'$ for all $t\in[0,1].$ In other words,
$G$ is a deformation retraction of $P(m)^2$ to $\tau'\in F_a.$ 

We claim that the restriction $G|_{[0,1]\times F_a}$ is a deformation retraction of $F_a$
to $\tau'\in F_a.$

Now assume $\psi'\in F_a.$ It follows that $\tilde{\psi}'\in F_a$ is of the form
$(\delta_{ij},0,\tilde{\psi}_{ijkl})$ as well as $\tilde{\tau}'=(\delta_{ij},0,\tilde{\tau}_{ijkl})\in F_a$
and $h(s,\psi')=(\delta_{ij},0,s\cdot\tilde{\tau}_{ijkl}+(1-s)\cdot\tilde{\psi}_{ijkl}).$

Suppose we have a Riemannian metric $g$ defined near $0\in\mb{R}^m$ which
satisfies $g_{ij}(0)=\delta_{ij}$ and $\p_k g_{ij}(0)=0.$ Then an elementary calculation 
shows
\[R_{ijks}(0)=
\frac{1}{2}\left(\p_j\p_k g_{si}-\p_j \p_s g_{ik}-\p_i\p_k g_{sj}+\p_i\p_s g_{jk} \right)(0).\]
Thus we conclude $h(s,\psi')\in F_a$ for all $s\in[0,1].$ Taking into account that the transformations
$L(t,\psi')$ are induced by coordinate changes we have 
\[G(t,\psi')=L(t,\psi')^{-1}(h(t,\psi'))\in F_a\]
for all $t\in[0,1].$ 

It follows that $F_a$ is contractible and by elementary obstruction theory there 
exists a global section of $q^2:\mc{R}_{0,a}\rightarrow M,$ i.e. a formal solution.
\end{proof}

\end{document}